\documentclass{amsart} 
\usepackage{amsmath,amsthm}
\usepackage{amssymb,amscd,latexsym}
\usepackage{boxedminipage}
\usepackage{eepic}
\usepackage{epic}
\usepackage{wrapfig}
\usepackage[dvips]{graphicx}

\newcommand{\erase}[1]{}
\newtheorem{theorem}{Theorem}[section]

\newtheorem{proposition}[theorem]{Proposition}

\newtheorem{remark}[theorem]{Remark}
\newtheorem{fact}[theorem]{Fact}
\newtheorem{lem}[theorem]{Lemma}
\renewcommand{\theenumi}{\roman{enumi}}
\renewcommand{\labelenumi}{$(\rm{\theenumi})$}
\numberwithin{equation}{section}
\newcommand{\bp}{\begin{pmatrix}}
\newcommand{\ep}{\end{pmatrix}}
\newcommand{\bps}{\begin{smallmatrix}}
\newcommand{\eps}{\end{smallmatrix}}

\def\R{{\mathbb R}}
\def\Z{{\mathbb Z}}

\def\O{{\mathcal O}}

\def \0{{\bf 0}}
\def \1{{\bf 1}}

\def \rank{\mathrm{rank}}

\def \mf#1#2#3#4{
\xymatrix{{#1}\  \ar@<0.4ex>[r]^{{#2}} & \ {#4}
\ar@<0.4ex>[l]^{{#3}}
}
}

\def \mfs#1#2#3#4{\!
\xymatrix@C=1.5em{{#1} \! \ar@<0.2ex>[r]^{{#2}} & \! {#4}
\ar@<0.2ex>[l]^{{#3}}
}
\!}

\def \mfl#1#2#3#4{
\xymatrix@C=2.6em{{#1}\  \ar@<0.4ex>[r]^{{#2}} &\  {#4}
\ar@<0.2ex>[l]^{{#3}}
}
}

\def \mfss#1#2#3#4{\!
\xymatrix@C=1.5em{{#1} \ar@<0.3ex>[r]^{{#2}} & {#4}
\ar@<0.3ex>[l]^{{#3}}
}
\!}

\begin{document}
\title{Zero loci of 
Skew-growth functions\\ 
 for dual Artin monoids.} 

\author{Tadashi Ishibe and  Kyoji Saito}

  







\maketitle

\begin{abstract} 
We show that the skew-growth function of a dual Artin monoid of finite type $P$ has exactly $\rank(P)=:l$ simple real zeros on the interval $(0, 1]$.  The proofs for types $A_l$ and $B_l$ are based on an unexpected fact that the skew-growth functions, up to a trivial factor, are expressed by Jacobi polynomials due to a Rodrigues type formula in the theory of orthogonal polynomials.  The skew-growth functions for type $D_l$ also satisfy Rodrigues type formulae, but the relation with Jacobi polynomials is not straightforward, and the proof is intricate. We show that the smallest root converges to zero as the rank $l$ of all the above types tend to infinity.
\end{abstract}

\section{Introduction}
We study the zero loci of the {\it skew-growth function} (\cite{[Sa2]})\footnote
{The study of the skew-growth function of a monoid and its zero loci is motivated by the study of the {\it partition functions associated with the monoid}, since the partition functions are given by certain residue formula at the zero loci of the skew-growth function (\cite{[Sa3]} \S11 Th.6 and \S12).
}
of a {\it dual Artin monoid} of finite type (\cite{[Be]})\footnote
{Dual Artin monoids were introduced by D. Bessis (\cite{[Be]}, c.f.\ \cite{[B-K-L]}) under the name {\it dual braid monoids}. For more detailed explanations, see the following part of \S1 and Appendix III.}. 
The skew-growth function is identified with the generating function of M\"obius invariants (called the {\it characteristic polynomial}) of the {\it lattice of non-crossing partitions} (\cite{[K],[Be],[B-W]}), and is further shown by several authors (\cite{[At],[A-T],[B-W],[Ch]}) to be equal to the generating function of dimensions of cones of the positive part of the {\it cluster fan} of Fomin-Zelevinsky (\cite{[F-Z1]}). We observe that this combinatorially defined function shows an unexpected strong connection with orthogonal polynomials. With the help of them, our goal is to show that the roots of the skew-growth function are simple and lying in the interval $(0, 1]$, and that  the smallest root converges to zero as the rank of the type tends to infinity. 

 Let us explain the contents. 
Recall (\cite{[B-S]}) that an Artin group  $G_M$ (resp.\ an Artin monoid $G_M^+$) associated with  a Coxeter matrix $M=(m_{ij})_{i,j\in I}$ is a group (resp.\ monoid) generated by letters $i$ ($i\!\in\! I$) and defined by the Artin braid relations:
 \begin{equation}
 \label{Artin relation}
\qquad \qquad \qquad 
a_ia_ja_i\cdots\ =\  a_ja_ia_j\cdots \qquad (i,j\in I)
\end{equation}
where both sides are words of alternating sequences of  
letters $a_i$ and $a_j$ of the same length $m_{ij}=m_{ji}\in\Z_{>0}$ with the initial letters $a_i$ and $a_j$, respectively. The natural morphism $G_M^+ \to G_M$ is shown to be injective (in  particular, $G_M^+$ is cancellative) so that the Artin monoid is regarded as a submonoid of the Artin group. 
Requiring more relations $a_ia_i=1$ ($i\in I$), we obtain the Coxeter group $\overline{G}_M$ and the quotient morphism:  
$\pi: G_M \rightarrow   \overline{G}_M$ (for short, we shall denote $\pi(g)$ by $\overline g$ for $g\in G_M$).
We call ${\bf S}:=\{a_i\mid  i\in I\}$ the simple generator system of the Artin group and the Artin monoid, and $S:=\{\overline {a}_i\mid  i\in I\}$ that of the Coxeter group. The number $l:=\#I=\#{\bf S}=\#S$ is called the {\it rank}. 
Since the relations (\ref{Artin relation}) are homogeneous with respect to the word length,  we define
the degree homomorphism:
\begin{equation}
\label{degree map}
\deg \ : \ \ G^+_M \rightarrow \Z_{\ge0}  \qquad \& \qquad  G_M \rightarrow \Z 
\end{equation}
by assigning the word length to each equivalence class of $G^+_M$, and extending it additively to $G_M$. An element of $G^+_M$ is of degree 0 if and only if it is the unit. 

%

In the present note, we study only the cases when the Coxeter group is of finite of type  $P$ for $P\in${\small $\{l\ (l\!\ge\!1), B_l\ (l\! \ge\!2), D_l\ (l\! \ge\! 4), E_l\ (l=6,7,8), F_4, G_2, H_3, H_4,$ $I_2(p)\ (p\ge3)\}$}. Then, we shall denote the Coxeter group and the Artin group (resp.\ monoid)  by $\overline{G}_P$ and also by $G_P$ (resp.\ $G^+_P$), and call them of type $P$. 

Depending on a choice of a Coxeter element $c_P:=\prod_{i\in I}a_i$, Artin groups were also studied using larger generating set $\bf T\supset \bf S$ 
\begin{equation}
\label{T}
{\bf T}:=\cup_{m\in \Z} (c_P)^m{\bf S}(c_P)^{-m},
\end{equation} 
by David Bessis (\cite{[Be]}, c.f.\ \cite{[B-K-L]} for type $A_l$ and \cite{[Br]} for three generators case). Actually, the projection $\pi$ induces a bijection from ${\bf T}$  to the set $T$ of all reflections (i.e.\ elements which are conjugate to elements in $S$) in the Coxeter group $\overline{G}_P$. It was shown that the generator system ${\bf T}$ satisfy
quadratic relations
\begin{equation}
\label{dual relation}
     r\cdot s =  s^*\cdot r
 \end{equation}
 for non-crossing pairs $r, s\in \bf T$ and $s^*\in \bf T$ with the relation $\bar{r}\cdot \bar{s} =  \bar{s^*}\cdot \bar{r}$ (see Appendx III), 
 and that the monoid generated by {\bf T} and defined by the relation \eqref{dual relation}:
 \begin{equation}
 \label{dual Artin}
 G^{dual +}_{P} \  := \   \langle \ \bf T \ \mid  \ \eqref{dual relation} \ \mathrm{for\ all\ non crossing\ pairs} \ \rangle_{monoid}
 \end{equation}
 is naturally embedded into the Artin group $G_P$. D.\ Bessis called $G^{dual +}_{P}$ a {\it dual braid monoid}. However, since we want to use the terminology "braid" only for type $A_l$ ($l\in\Z_{\ge1}$) when the Artin group $G_{A_l}$ is a braid group $B(l+1)$, we modify the terminology in the present note, and call the monoid $G^{dual +}_{P}$ a {\it dual Artin monoid} of type $P$, and regard it as a submonoid of the Artin group $G_P$.
%
By definition \eqref{T}, we have $\deg(t)=1$ for all $t\in \bf T$, and hence the restriction of the degree morphism \eqref{degree map} induces a morphism $\deg : G^{dual +}_{P} \rightarrow \Z_{\ge0}$ such that an element of $G^{dual +}_P$ is of degree 0 if and only if it is the unit element.

For any cancellative monoid $M$ with a degree homomorphism $\deg\! :\! M\to\! \R_{\ge0}$ satisfying an ascending chain condition, we  introduced in \cite{[Sa2]} the concept of a {\it skew-growth function} $N_{M,\deg}(t)$ in the Novikov ring $\Z[[t^\mathrm{range(\deg)}]]$   
(recall Footnote 1 for its motivation). In particular, if the monoid $M$ is a lattice (i.e. $M$ admits a left or right least common multiple for any finite subsets of $M$), the skew-growth function is given by a simple formula:
\begin{equation}
\label{skew-growth}
N_{M,\deg}(t) = \sum_{J\subset I}(-1)^{\# J} t^{\deg(lcm_l(J))}
\end{equation}
where the index $J$ runs over all finite subsets of $I:=M\setminus\{1,M\cdot M\}$ (=the set of minimal generators of $M$). 
It is easy to see $(1-t)\mid N_{M,\deg}(t)$, and we set the {\it reduced skew-growth function}: 
$ 
\widehat{N}_{M,\deg}(t):=N_{M,\deg}(t)/(1-t)
$.


Actually, the poset structures on Artin monoids and dual Artin monoids w.r.t. left (or right) division relations,  are known  (\cite{[B-S]}, \cite{[Be]}) to be lattices so that the formula \eqref{skew-growth} is valid, where the minimal generator system $I$ is $\bf S$ or $\bf T$, respectively. Furthermore, these monoids are constructed from a certain 
 finite pre-monoid $[1,\Delta]_S$ and $[1, c]_T$ in the Coxeter group (Bessis \cite{[Be]}, see Footnote 3 and Appendix III),  
  where the poset structure on $[1, c]_T$ is called the {\it lattice of non-crossing partitions of type} $P$ (\cite{[Be]}, \cite{[B-W]}). Then, we observe that the characteristic polynomial (i.e.\ the generating function of the M\"obius function of the finite lattices  $[1,\Delta]_S$ and $[1, c]_T$ (\cite{[St]} (3.39)) coincides with the skew-growth function \eqref{skew-growth} for Artin and dual Artin monoids, respectively (\cite{[C-F],[Sa2]}, see Appendix III, Fact \ref{skew=character}).  
%

Explicit expressions of the skew-growth function for the Artin monoid $G_P^+$ are given by calculating the word length of fundamental elements for all sub-diagrams of Dynkin diagrams of type $P$  (\cite{[A-N]},\cite{[Sa1]}).  To obtain an explicit description of the skew-growth function for the dual Artin monoid $G_P^{dual+}$, we need more considerations: Chapoton (\cite{[Ch]}) conjectured a transformation formula from the generating function of two variables of M\"obius invariants of pairs of elements of the lattice $[1,c_P]_T$ to the two variable generating function of cone-counting of the cluster fan $\Delta(P)$ (Fomin-Zelevinsky \cite{[F-Z1]}). In particular, the specialization of the Chapoton formula to one variable is 
the identity between the  generating function of M\"obius numbers of the lattice $[1,c_P]_T$ with the generating function of the cones of the positive part $\Delta_+(P)$ of the fan, and  was already shown by Chapoton ([ibid])  (the general formula was proven by Athanasiadis, Brady and Watt \cite{[B-W]} \cite{[A-B-W]} depending on each of the cases separately, and a case free proof is given by Athanasiadis \cite{[At]}). Thus, the $k$th coefficient of the skew-growth function $N_{{G_P^{dual+}},\deg}(t)$, up to sign,  is equal to the number of $k$-dimensional cones of the cluster fan $\Delta_+(P)$ of type $P$.
In {\bf Tables A} and {\bf B} of Appendix I, we list explicit formulae of skew-growth functions of dual Artin monoids of finite type.

\medskip
We are now interested in the zero loci of the skew-growth functions $N_{G_P^+}(t)$ and $N_{G_P^{dual+}}(t)$ of finite type $P$.  For the Artin monoid cases, suggested by some numerical experiments,  the following 1, 2 and 3 were conjectured  in \cite{[Sa1]}. 

\medskip
\noindent
1.\ $\widehat{N}_{G^+_P}(t) = N_{G^+_P}(t)/(1\!-\!t)$ is an irreducible polynomial over $\Z$. 

\noindent
2.\ There are $\rank(P)$  simple roots of $N_{G^+_P}(t)$ on the interval $(0, 1]$. 

\noindent
3.\ The smallest real root is strictly less than the absolute values of any other roots. 

\medskip
\noindent
Furthermore, the smallest real root of $N_{G^+_P}(t)$ seems to be decreasing and  convergent to a  constant $0.30924...$ as the rank $l$ tends to infinity.  

\smallskip
In analogy with 1, 2 and 3, and also inspired by some numerical experiments for the dual Artin monoids (see Remark 1.2 and Appendix II for the figures of the zero loci of the functions of types $A_{20},\ B_{20},\ D_{20}$ and $E_8$), we conjecture the following. 

\medskip
\noindent
{\bf Conjecture 1.}  $\widehat{N}_{G^{dual +}_{P}}(t) = N_{G^{dual +}_{P}}(t)/(1-t)$ is an irreducible polynomial over $\Z$, up to the trivial factor $1-2t$ for the types $A_l$ ($l$: even) and $D_4$ (see {\bf Fact}s at the end of \S3 and \S5 for the factor $1-2t$).

\smallskip
\noindent
{\bf Conjecture 2.}  $N_{G^{dual +}_{P}}(t)$ has  $l=\rank(P)
$ simple real roots on the interval $(0, 1]$, including a simple root at $t=1$. 

\smallskip
\noindent
{\bf Conjecture 3.}  The smallest root of $N_{G^{dual +}_{P}}(t)$ decreases and converges to $0$ 
as the rank $l$ tends to infinity for the infinite series of type $A_l$, $B_l$ and $D_l$. 

\begin{remark}
By definition, the degree of $N_{G^{dual +}_{P}}(t)$ is equal to the word length of the Coxeter element $c_P=\#{\bf S}=\rank(P)=l$. Therefore,  Conjecture 2. implies that all roots of $N_{G^{dual +}_{P}}(t)$ are on the interval $(0, 1]$. 
\end{remark}
\begin{remark}
Conjecture 1. is approved  for types $A_l$   ($1\le l \le 30$), 
$B_l$  ($2\le l \le 30$),  $D_l$   ($4\le l \le 30$), $E_6$, $E_7$, $E_8$, $F_4$, $G_2$, $H_3$, $H_4$ and $I_2(p)$ ($p\ge 3$)  by using the software package Mathematica on the Table A and B in Appendix I.  
\end{remark}

The goal of the present note is to give affirmative answers to Conjectures 2 and 3.  The proof is based on the explicit expressions of the skew-growth functions in Appendix I, and is divided into three groups: i) two infinite series of types $A_l$ and $B_l$, ii) the infinite series of type $D_l$, and iii) the remaining exceptional types $E_6,\ E_7,\ E_8,\ F_4,\ G_2$ and non-crystallographic types $H_3,\  H_4$ and $I_2(p)$.  

Let us give a review of the proof.  The third group iii) consists only of types of bounded ranks  so that Conjecture 3 does have no meaning, and the Conjecture 1 and 2 are verified by direct calculations for each type. Therefore, our main task is to manage the infinite series with the growing rank $l$ in groups i) and ii).  

The key fact to do this is a rather mysterious expression for the skew-growth functions, which we shall call {\it Rodrigues type formula} in analogy with the Rodrigues formula in orthogonal polynomial theory (\cite{[Sz]}). 
Namely, in \S2,  we show that the skew-growth functions of rank $l$ is expressed (up to a simple constant factor or linear combinations) by a polynomial of the form 
$\big(\frac{d}{dt}\big)^{l+\varepsilon_1}\big[t^{l+\varepsilon_2}(1-t)^{l+\varepsilon_3}\big]$ where $\varepsilon_i$ are some small and fixed ``fluctuation" numbers. Its proof is elementary, however the meaning of the expression, in particular, of the fluctuation numbers, is still unclear. 

Then, we show further in \S3 that the Rodrigues type formulae lead to recursion relations of the series of the  skew-growth functions. Namely, we obtain 3-term recursion relations for the series of types $A_l$ and $B_l$, respectively, and 4-term recurrence relations for the series of type $D_l$.  In case of types $A_l$ and $B_l$, we may reduce the proof of recurrence to that of corresponding Jacobi polynomials, but we don't know whether such type of reduction is possible for type $D_l$ or not.

Proofs of Conjecture 2 for types $A_l$ and $B_l$ are given in \S4. A direct proof is that Jacobi polynomial expressions of the skew-growth functions for the types in \S2 imply automatically that Conjecture 2 is true. An alternative approach is that the recurrence relations in \S3 show easily that  the series of the skew-growth functions form a Sturm sequence in  the sense of \cite{[T]} Theorem 4.3.
Proof of Conjecture 2 for type $D_l$ is more complicated and is given in \S5, where we essentially use the Rodrigues type formula in \S2  but no explicit use of Jacobi polynomials. We do not know whether there is a proof to reduce the conjecture to Jacobi polynomials.


In \S6, we prove Conjecture 3 affirmatively, where the relationship of the skew-growth functions with Jacobi polynomials given in \S2 plays the key role. 
For series of types $A_l$ and $B_l$, we have two proofs again. A direct proof is based on the general fact that the zeros of a series of Jacobi polynomials is dense in the interval (0,1), implying Conjecture 3. Another alternative proof is based on a  sharp approximation of the distribution of all roots (the ``density" is proportional to {\small $1/\sqrt{t(1-t)}$}) of the skew-growth function of type $A_l$ and $B_l$ and  is given by sandwiching the roots by the roots of Legendrian polynomials.  The proof for the series $D_l$ uses again Rodrigues type formula in \S2, where the functions of type  $D_l$ are expressed by those of type $B_l$ so that the roots of type $D_l$ are sandwiched by the roots of type $B_l$. 

Appendix I gives Tables A and B of skew-growth-functions of dual Artin monoids of finite type.
Appendix II exhibits figures of the zeros of the skew-growth functions for dual Artin monoids of type ${20}$, $B_{20}$, $D_{20}$ and $E_8$. 
In Appendix III, we recall Bessis's study of dual Artin monoids, and then identify the skew-growth function of them with the characteristic polynomial of the non-crossing partition lattice. 

\begin{remark} {\rm Finally, let us give a rather vague philosophical
 remark on the present study.  Recall that our starting point was the skew-growth function of a dual Artin monoid.
That is,  the starting function is given combinatorially by the enumeration of the dimensions of cones of a cluster fan. Then, it turns out to be (unexpectedly) expressed by certain analytic objects such as Jacobi orthogonal polynomials. This picture resembles 
mirror symmetry: some enumeration of BPS states are mirror to some Hodge structure or to period integral theory. Both have the pattern that an enumeration of some combinatorial objects is transformed to a function of analytic nature. 
However, we have no further explanations of this analogy.}
\end{remark}

\bigskip
\section{Rodrigues type formulae and orthogonal polynomials}

In this section, for the three infinite series $l,\ B_l$ and $D_l$ of skew-growth functions, we show two facts: Rodrigues type formulae (Theorem 2.1).

\begin{theorem} {\bf (Rodrigues type formula) }
\label{Rodrigues} 
 For types $l \ (l\ge1)$, $B_l\ (l\ge2)$ and $D_l \ (l\ge4)$, we have the formulae:
{\small
\begin{align}
\label{RodriguesA}
tN_{G^{dual +}_{{A_l}}}(t) &= \frac{1}{l!} \frac{\mathrm{d}^{l-1}}{\mathrm{d}t^{l-1}}\biggl[ t^l(1-t)^l \biggr], \\
\label{RodriguesB}
N_{G^{dual +}_{B_{l}}}(t) & =  \frac{1}{(l-1)!} \frac{\mathrm{d}^{l-1}}{\mathrm{d}t^{l-1}}\biggl[ t^{l-1}(1-t)^l \biggr], \\
\label{RodriguesD}
N_{G^{dual +}_{D_{l}}}(t) &= \frac{1}{(l-2)!}\frac{\mathrm{d}^{l-2}}{\mathrm{d}t^{l-2}}\biggl[ t^{l-2}(1-t)^{l} \biggr] + \frac{1}{(l-3)!}\frac{\mathrm{d}^{l-3}}{\mathrm{d}t^{l-3}}\biggl[ t^{l-1}(1-t)^{l-2} \biggr]\\
 &= \frac{1}{(l-2)!}\frac{\mathrm{d}^{l-3}}{\mathrm{d}t^{l-3}}\biggl[ t^{l-3}(1-t)^{l-2} \Bigl\{(l-2) - (3l-4)t + (3l-4)t^2 \Bigr\} \biggr].
\end{align}
}
\end{theorem}
\begin{proof} 
Type $A_l$:  The right hand side (\ref{RodriguesA}) is calculated as
{\footnotesize
\[
 \frac{1}{l!} \frac{\mathrm{d}^{l-1}}{\mathrm{d}t^{l-1}}\biggl[ t^l(1-t)^l \biggr]= 
\frac{1}{l!} \frac{\mathrm{d}^{l-1}}{\mathrm{d}t^{l-1}}\biggl[\sum_{k=0}^{l} (-1)^{k} \binom {l}{k} t^{l+k} \biggr]
= t\sum_{k=0}^{l}(-1)^{k}\frac{(l+k)!}{(l-k)!k!(k+1)!}t^k.
\]}
This gives, up to a factor $t$, RHS of the expression of $N_{G^{dual +}_{{A_l}}}(t)$ in Table A.\\

\smallskip
\noindent
Type $B_l$:  The right hand side of (\ref{RodriguesB}) is calculated as 
{\footnotesize
\[
\frac{1}{(l-1)!} \frac{\mathrm{d}^{l-1}}{\mathrm{d}t^{l-1}}\biggl[ t^{l-1}(1-t)^l\biggr] = 
\frac{1}{(l-1)!} \frac{\mathrm{d}^{l-1}}{\mathrm{d}t^{l-1}}\biggl[\sum_{k=0}^{l} (-1)^{k} \binom {l}{k} t^{l+k-1}\biggr]= \sum_{k=0}^{l}(-1)^{k}\frac{l(l+k-1)!}{(l-k)!k!k!}t^k.
\]}
This gives RHS of the expression of $N_{G^{dual +}_{B_{l}}}(t)$ in Table A.\\

\smallskip
\noindent
Type $D_l$:  We compute the right hand side of (\ref{RodriguesD}).
{\footnotesize\[
\ \  \frac{1}{(l-2)!} \frac{\mathrm{d}^{l-2}}{\mathrm{d}t^{l-2}}\biggl[\sum_{k=0}^{l} (-1)^{k} \binom {l}{k} t^{l+k-2} \biggr] + \frac{1}{(l-3)!}\frac{\mathrm{d}^{l-3}}{\mathrm{d}t^{l-3}}\biggl[\sum_{k=0}^{l-2} (-1)^{k} \binom {l-2}{k} t^{l+k-1} \biggr]
 \]}
\vspace{-0.1cm}
{\footnotesize\[
= \sum_{k=0}^{l}(-1)^{k}\frac{l(l-1)(l+k-2)!}{(l-k)!k!k!}t^k + \sum_{k=0}^{l-2}(-1)^{k}\frac{(l-2)(l+k-1)!}{(l-2-k)!k!(k+2)!}t^{k+2} \,\,\,\,\,\,\,\,\,\,\,\,\,\,\,\,\,\,\,\,\,\,\,\,\,\,\,
\]}
\vspace{-0.1cm}
{\footnotesize\[
= \sum_{k=0}^{l}(-1)^{k}\frac{l(l-1)(l+k-2)!}{(l-k)!k!k!}t^k + \sum_{k=2}^{l}(-1)^{k}\frac{(l-2)(l+k-3)!}{(l-k)!k!(k-2)!}t^{k} \,\,\,\,\,\,\,\,\,\,\,\,\,\,\,\,\,\,\,\,\,\,\,\,\,\,\,\,\,\,\,\,\,\,\,\,\,\,\,\,\,
\]}
\vspace{-0.1cm}
{\footnotesize\[
= \sum_{k=0}^{l} (-1)^{k} \biggl( \binom {l}{k}\binom {l+k-2}{k} + \binom {l-2}{k-2}\binom {l+k-3}{k}\biggr) t^k.\,\,\,\,\,\,\,\,\,\,\,\,\,\,\,\,\,\,\,\,\,\,\,\,\,\,\,\,\,\,\,\,\,\,\,\,\,\,\,\,\,\,\,\,\,\,\,\,\,\,\,\,\,\,\,\,\,\,\,\,\,\,\,\,\,\,
\vspace{-0.1cm}
\]}
This gives RHS of the expression of $N_{G^{dual +}_{D_{l}}}(t)$ in Table A.
\end{proof} 

For $l\in\Z_{\ge0}$ and $\alpha,\beta\in \R_{>-1}$, let $P^{(\alpha, \beta)}_{l}(x)$ be the Jacobi polynomial (c.f.\ \cite{[Sz]} 2.4). Let us introduce the  {\it shifted Jacobi polynomial} of degree $l$ by setting
\[
 \widetilde{P}^{(\alpha, \beta)}_{l}(t) := P^{(\alpha, \beta)}_{l}(2t-1) . 
\]
\begin{fact}\cite{[Sz]}(4.3.1) 
\label{Jacobi}
The shifted Jacobi polynomial 
satisfies the following equality
\[
\begin{array}{c}
(t-1)^{\alpha}t^{\beta}\widetilde{P}^{(\alpha, \beta)}_{l}(t) = \frac{1}{l!} \frac{\mathrm{d}^{l}}{\mathrm{d}t^{l}}\big[ (t-1)^{l + \alpha}t^{l + \beta} \big].
\end{array}
\]
\end{fact}

Comparing two formulae in Theorem \ref{Rodrigues} and Fact \ref{Jacobi}, we obtain expression of the skew-growth functions for types $A_l$, $B_l$ and $D_l$ by shifted Jacobi polynomials.
\begin{align}
\label{JacobiA}
N_{G^{dual +}_{{A_l}}}(t) & \ =\  \frac{(-1)^{l-1}}{l}(1-t)\widetilde{P}^{(1, 1)}_{l-1}(t),\\
\label{JacobiB}
N_{G^{dual +}_{B_{l}}}(t) & \ =\  (-1)^{l-1}(1-t)\widetilde{P}^{(1, 0)}_{l-1}(t),\\
\label{JacobiD}
N_{G^{dual +}_{D_{l}}}(t) & \ = \ (-1)^{l-2}(1-t)^2\widetilde{P}^{(2,0)}_{l-2}(t)+
(-1)^{l-1}t^2(1-t)\widetilde{P}^{(1,2)}_{l-3}(t).
\end{align}

\begin{remark} There are Jacobi polynomial expressions for types $H_3$ and $I_2(p)$:
\[ 
\begin{array}{rcl}
N_{G^{dual +}_{H_{3}}}(t) &\! \!=\!\! &\frac{4}{3}t^{1/2}\cdot (\frac{d}{dt})^2\big(t^{3/2}(1-t)^3\big)= \frac{3}{8}(1-t) \widetilde{P}_2^{(1,-1/2)}(t) \\
N_{G^{dual +}_{I_{2}(p)}}(t) &\! \! =\! \!& \frac{1}{(1+b)t^b(1-t)^{a-1}}\frac{d}{dt}\big(t^{1+b}(1-t)^{1+a}\big) = \frac{1-t}{1+b}\  \ \widetilde{P}_1^{(a,b)}(t) \
\end{array}
\] 
where  $a,b\in \R_{>-1}$ such that $1+a=(p-2)(1+b)$. But we shall not use them.
\end{remark}

\bigskip

\section{Recurrence relations for types $A_l$ ($l\ge1$), $B_l$ ($l\ge2$) and $D_l$
($l\ge4$)}

As an application of the Rodrigues type formulae, we show that the the series of skew-growth functions for types $A_l$ ($l\ge1$), $B_l$ ($l\ge2$) and $D_l$ satisfy either 3-term or 4-term recurrence relations (Theorem 3.1). 

\medskip
\begin{theorem}   
 For type $A_l$ and $B_l$, the following $3$-term recurrence relation holds. 
{\small
\begin{equation}
\label{recurrenceAB}
\begin{array}{rl}
\!\!\! (l+3)N_{G^{dual +}_{{l+2}}}(t) \!\!\! &= -(2l+3)(2t-1) N_{G^{dual +}_{{l+1}}}(t) - l N_{G^{dual +}_{A_l}}(t). \\
\!\!\! (l+2)N_{G^{dual +}_{B_{l+2}}}(t) \!\!\! & =
 -(2l+3)\biggl\{2t-\frac{2(2l^2+4l+1)}{(2l+1)(2l+3)}\biggr\} N_{G^{dual +}_{B_{l+1}}}(t) - \frac{l(2l+3)}{2l+1}N_{G^{dual +}_{B_{l}}}(t)
\end{array}
\end{equation}  
}
For type $D_l$, the following $4$-term recurrence relation holds.
{\small
\begin{equation}
\label{recurrenceD}
N_{G^{dual +}_{D_{l+3}}}(t)  = (a_{l} + b_{l} t) N_{G^{dual +}_{D_{l+2}}}(t) 
 +(c_{l} + d_{l} t + e_{l} t^2) N_{G^{dual +}_{D_{l+1}}}(t) + (f_{l} + g_{l} t)  N_{G^{dual +}_{D_l}}(t). 
 \end{equation}
}
Here, $a_{l}$, $b_{l}$, $c_{l}$, $d_{l}$, $e_{l}$, $f_{l}$ and $g_{l}$ are the following rational functions:
 { \small
 \[
  a_{l} = \frac{(l + 2)(43l^3- 78l^2- 129l-24)}{(l + 3)(43l^3- 35l^2- 36l-32)},\,\,\,\,\,\,\,\,\,\,\,
  \]
  \vspace{0.05cm}
  \[
  b_{l} = -\frac{86l^4+ 145l^3- 196l^2- 623l-456}{(l + 3)(43l^3 - 35l^2- 36l - 32)},
  \]
   \vspace{0.05cm}
 \[
  c_{l} = \frac{l(43l^3+ 180l^2+ 45l+56)}{(l+3)(43l^3- 35l^2- 36l-32)},\,\,\,\,\,\,\,\,\,\,\,\,
  \]
  \vspace{0.05cm}
  \[
  d_{l} = -\frac{2l(172l^3+ 333l^2- 23l-32)}{(l+3)(43l^3- 35l^2- 36l-32)},\,\,\,\,\,\,\,\,
  \]
  \vspace{0.05cm}
  \[
  e_{l} =\frac{2(2l-1)(2l+1)(43l^2 + 51l - 24)}{(l+3)(43l^3- 35l^2- 36l-32)},\,\,\,\,
 \]
   \vspace{0.05cm}
 \[
  f_{l} = -\frac{(l-1)(43l^3+ 137l^2+ 38l-48)}{(l+3)(43l^3- 35l^2- 36l-32)},\,\,\,\,\,
  \]
  \vspace{0.05cm}
 \[
  g_{l} =\frac{(l-1)(2l+1)(43l^2+ 51l-24)}{(l+3)(43l^3- 35l^2- 36l-32)}.\,\,\,\,\,\,\,\,\,\,
  \]
  }

\end{theorem}
\begin{proof}

The relations for  types $A_l$ and $B_l$ are shown either directly using the explicit formule \eqref{RodriguesA} and \eqref{RodriguesB}, or, in view of \eqref{JacobiA} and \eqref{JacobiB}, reducing to the relations for corresponding Jacobi polynomials (\cite{[Sz]}(4.5.1)). So, we have only to prove the relation for type $D_l$.  


Let us consider the $k$th coefficient of $N_{G^{dual +}_{D_{l}}}(t)$ up to the sign $(-1)^k$:
\[
\begin{array}{rcl}
\mathcal{C}(l, k)  &:= &\binom {l}{k}\binom {l+k-2}{k} + \binom {l-2}{k-2}\binom {l+k-3}{k}\\
& =&  \frac{(l+k-3)!}{(l-k)!(k!)^2}\big\{  l(l-1)(l+k-2)+(l-2)k(k-1) \big\}.
\end{array}
\]
We compute the coefficient of the term $(-t)^k$ on the right hand side of \eqref{recurrenceD}.
 {\small
 \[
\begin{array}{rl}
& a_{l}\cdot \mathcal{C}(l+2, k)- b_{l}\cdot \mathcal{C}(l+2, k-1)+ c_{l}\cdot \mathcal{C}(l+1, k)- d_{l}\cdot \mathcal{C}(l+1, k-1) \\
\vspace{0.2cm}
& + e_{l}\cdot \mathcal{C}(l+1, k-2) + f_{l}\cdot \mathcal{C}(l, k) - g_{l}\cdot \mathcal{C}(l, k-1) \\
 = & \large \frac{(l+k-4)!}{(l+3-k)!(k!)^2(l+3)(43l^3- 35l^2- 36l-32)}  \Big\{ -k^{2}(l-1)(l-k+2)(l-k+3) \\
& \ \ \ \times(2l+1)(43l^{2}+51l-24)(-4 + 6k - 2 k^{2} + 5l - 4 kl + k^{2}l - 4 l^{2} + k l^{2} +  l^{3}) \\ 
\vspace{-0.15cm}
\\
&  + \ \ \ 2(k-1)^{2}k^{2}(2l-1)(2l+1)(-24 + 51l+ 43l^{2}) \\
& \ \ \ \times(-6 + 5k-k^{2}+3l-4kl+k^{2}l - 2 l^{2} + k l^{2}+l^{3}) \\
\vspace{-0.15cm}
\\
& +\ \ \ (l+2)(l-k+3)(l+k-3)(l+k-2)(l+k-1)(43l^{3}- 78 l^{2}- 129l-24) \\
 & \ \ \ \times(2k + 2l + 2 kl + k^{2} l + 3 l^{2} + k l^{2} + l^{3}) \\
\vspace{-0.15cm}
 \\
& -\ \ \ (l-1)(l-k+1)(l-k+2)(l-k+3)(l+k-3)(43l^{3}+137l^{2}+38l-48)  \\
& \ \ \ \times(2k-2k^{2}+2l-2kl + k^2 l - 3l^{2}+kl^{2}+l^{3}) \\
\vspace{-0.15cm}
\\
&  +\ \ \ l(l-k+2)(l-k+3)(l+k-3)(l+k-2)(43l^{3}+ 180l^{2}+ 45l+56)  \\
& \ \ \ \times(k-k^{2}-l+k^{2}l + kl^{2}+l^{3}) \\
\vspace{-0.2cm}
\\
&  +\ \ \ 2k^{2}l(l-k+3)(l+k-3)(172l^{3}+ 333l^{2} - 23l-32) \\
& \ \ \ \times(-2 + 3k- k^{2}- 2kl+ k^{2}l-l^{2} + kl^{2} + l^{3}) \\
\vspace{-0.15cm}
\\
&  +\ \ \ k^{2}(l+k-3)(l+k-2)(86l^{4}+ 145 l^{3}- 196 l^{2}- 623l-456) \\
\vspace{0.3cm}
& \ \ \ \times(-2 + 2k + l + k^{2}l+2l^{2} + kl^{2}+l^{3})\Big\}  \\
\vspace{0.3cm}
=& \frac{(l+k)!}{(l+3-k)!(k!)^2} (6+5k+ k^{2}+11l+ 4kl+k^{2}l+6l^{2}+kl^{2}+l^{3}) \\
=& \mathcal{C}(l+3, k).
\end{array}
\]}

\end{proof}

As an application of the recurrence relation, we observe the following.

\medskip
\noindent
{\bf Fact.} The skew-growth function $N_{G^{dual +}_{{A_l}}}(t)$ ($l\ge1$) is divisible by $2t-1$ if and only if $l$ is even.

\section{Proof of Conjecture 2 except for types $D_l$ }

In the present section, we prove,  except for types $D_l$, following Theorem, which approves Conjecture 2.  The proof for types $D_l$ is given in the next section 5. 

\medskip
\begin{theorem} 
\label{rootsP}
 The skew-growth function $N_{G^{dual +}_{P}}(t)$ for any finite type $P$ has $\rank(P)$ simple roots on the interval $(0, 1]$, including a root at $t=1$. 
\end{theorem}

\begin{proof}
{\bf Case I:  type} $A_l$ ($l\in \Z_{\ge1}$) and $B_l$ ($l\in\Z_{\ge1}$).   

This is an immediate consequence of the formulae \eqref{JacobiA} and \eqref{JacobiB}, since the Jacobi polynomials $\widetilde{P}_{l-1}^{(1,1)}$ and $\widetilde{P}_{l-1}^{(1,0)}$ are well known to have $l-1$ simple roots on the interval $(0,1)$  (see \cite{[Sz]} Theorem 3.3.1).

\medskip
\noindent
{\bf Case II: Exceptional types and non-crystallographic types}    

Recall $\widehat{N}_{G^{dual +}_{P}}(t):=N_{G^{dual +}_{P}}(t)/(1-t)$, which is a polynomial of degree $l-1$. 
Then, we apply the Euclid division algorithm for the pair of polynomials $f_0:=\widehat{N}_{G^{dual +}_{P}}$ and $f_1:=(\widehat{N}_{G^{dual +}_{P}}(t))'$. So, we obtain, a sequence $f_0,\ f_1,\ f_2,\cdots$ of polynomials in $t$ such that $f_{k-1}=f_k\cdot q_{k-1}+f_{k+1}$ for $k=1,2,\cdots$ (where $q_{k-1}$ is the quotient and $f_{k+1}$ is the remainder). 

Then, we prove the following fact by direct calculations case by case.

\begin{fact} 
i) The degrees of the sequence $f_0,\ f_1,\ f_2,\cdots$ of polynomials descend one by one, and $f_{l-1}$ is a non-zero constant.

ii) The  sequence $f_0(0), f_1(0), -f_2(0), \ldots, -f_{l-1}(0)$ has constant sign and the sequence $f_0(1), f_1(1), -f_2(1), \ldots, -f_{l-1}(1)$ has alternating sign.
\end{fact}

Applying the Sturm Theorem (see for instance \cite{[T]} Theorem 3.3), we observe that 
$f_0$ has $l-1$ distinct roots on the interval $(0,1)$. Since the polynomial $f_0=\widehat{N}_{G^{dual +}_{P}}$ is of degree $l-1$, all the roots should be simple.

This completes a proof of Theorem \ref{rootsP}.
\end{proof}

\begin{remark}
An alternative proof of Theorem 3.1 for types $A_l$ and $B_l$ is given as follows: using the recurrence relations \eqref{recurrenceAB}, we see that the sequences {\footnotesize $\widehat{N}_{G^{dual +}_{A_l}}(t)$} and {\footnotesize $\widehat{N}_{G^{dual +}_{B_l}}(t)$} form Sturm sequences on the interval $[0,1]$ (see \cite{[T]} Theorem {\rm 3.3}).  
Then the number of sign changes of the boundary values of the sequences is counted  as $l-1$ by the facts: 
{\small $\widehat{N}_{G^{dual +}_{P}}(0) = 1$ and $\widehat{N}_{G^{dual +}_{P}}(1) = -N'_{G^{dual +}_{P}}(1)$, and 
$
N'_{G^{dual +}_{{A_l}}}(1) = (-1)^l, \
N'_{G^{dual +}_{B_{l}}}(1) = (-1)^{l}l \ \ \text{and}\ \
N'_{G^{dual +}_{D_{l}}}(1) = (-1)^{l}(l-2),
$
}
(use \eqref{recurrenceD} again). 
\end{remark}

\begin{remark}
In \cite{[I]},  the first author calculated the derivative at $t = 1$ of the skew-growth functions $N_{G_{P}^{+}}(t)$ for Artin monoids $G^{+}_{P}$. We remark that, for any type $P$, the following equality holds
\vspace{-0.3cm}
\[
N'_{G^+_{P}}(1) = N'_{G^{dual +}_{P}}(1).
\]
\end{remark}

\bigskip

\section{Proof of Conjecture 2 for types $D_l \ \ (l\ge4)$}

In this section, we prove the following theorem, which answers to Conjecture 2 for the types $D_l$ ($l\ge4$) affirmatively.

\begin{theorem}
\label{rootsD}
The polynomial $N_{G^{dual +}_{D_{l}}}(t)$ has $l$ simple roots on the interval $(0, 1]$.
\end{theorem}
\begin{proof}
Recall the Rodrigues type formula \eqref{RodriguesD} for type $D_l$ for $l \in \Z_{\ge4}$. Up to the factor $(1-t)$,  we consider the factor in the derivatives:
\[
\begin{array}{rcl}
H_{l}(t) &:=&  (t^2-t)^{l-3} \Bigl\{(l-2) - (3l-4)t + (3l-4)t^2 \Bigr\}\\
\end{array}
\]
so that the following equality holds.
\begin{equation}
\label{RodriguesD3}
\begin{array}{rcl}
N_{D_l}(t)=\frac{(-1)^{l-3}}{(l-2)!}\big(\frac{d}{dt}\big)^{l-3} \big( (1-t)H_l(t)\big)
\end{array}
\end{equation}

Set $H^{(i)}_{l}(t) := \frac{\mathrm{d}^{i}}{\mathrm{d}t^{i}}H_{l}(t)$ for $0\le i\le l-3$. Applying $i$-times the Rolle theorem to the polynomial $H_l(t)$, we know that $H^{(i)}_{l}(t)$ has at least $i$ number of distinct (possibly multiple) roots on the interval $(0,1)$ and that if the number of roots is exactly equal to $i$ then the function $H^{(i)}_{l}(t)$  changes its sign at the zeros. 

\begin{lem}
\label{rootsH}
The polynomial $H_l^{(l-4)}(t)$ has $l$ simple roots on the interval $[0,1]$.
\end{lem}
\noindent
{\it Proof of } {\bf Lemma \ref{rootsH}.}
Since $H_{l}(t)$ is invariant by the reflection $t\mapsto 1-t$, we have 
$H^{(i)}_{l}(1/2+t)=(-1)^{i}H^{(i)}_{l}(1/2-t)$. 
Therefore, the set of roots of $H^{(i)}_{l}$ are symmetric with respect to the reflection centered at $t=1/2$. In particular, for $0<2L\le l-3$, if $H^{(2L)}_{l}(1/2)$ is non 
zero, then the half of real roots are lying on the half line $(1/2,\infty)$. If the number of roots on the interval $(1/2,1)$ were equal to $L$ (hence $2L$ roots on $(0,1)$), as we saw above by Rolle's theorem, the sign of 
$H^{(2L)}_{l}(1/2)$ should be $(-1)^{l+L-3}$.  However, the following {\bf Formula A}\ shows that  it is not the case for some $L$. That is, $H^{(2L)}_{l}(t)$ has more than $L$ roots on the interval $(1/2,1)$. Since $\deg(H^{(2L)}_{l})=2l-2L-4$ and the multiplicity of the zeros at $t= 0, 1$ is $l-3-2L$, we conclude that the number of roots on the interval (1/2,1) is bounded by, and, hence is equal to $L+1$, and that all roots are simple. 
Then, for $2L\le i\le l-3$, applying $i-2L$ times Rolle's theorem to $H_l^{(2L)}$, we see that the polynomial $H^{(i)}_{l}$ has $i+2$ simple roots on the interval $(0,1)$. 
Therefore, it remains only to show the following formula.

\bigskip
\noindent
{\bf Formula A.}  According to the residue class $l\bmod 4$, we have 
\[
\begin{array}{rcl}
\vspace{0.2cm}
H_{4L}^{(2L)}(\frac{1}{2}) & =& (-1)^{L}2^{6-6L}\frac{(4L-2)!(2L)!}{(3L-2)!L!},\\
\vspace{0.2cm}
H_{4L+1}^{(2L)}(\frac{1}{2}) & = & (-1)^{L+1}2^{2-6L}3\frac{(4L-1)!(2L)!}{(3L-1)!L!}, \\
\vspace{0.2cm}
H_{4L+2}^{(2L)}(\frac{1}{2}) & = &(-1)^{L}2^{1-6L}\frac{(4L)!(2L)!}{(3L)!L!}, \\
H_{4L+3}^{(2L)}(\frac{1}{2}) & = &(-1)^{L+1}2^{-2-6L}\frac{(4L+1)!(2L)!}{(3L+1)!L!}.\qquad\qquad
\end{array}
\]

\noindent
Before showing Formula A, we first prepare an auxiliary Formula B.

\noindent
{\bf Formula B.} Set $h_k^{(i)}(t):=\big(\frac{d}{dt}\big)^i(t(t-1))^k$ for $0\le i \le k$.  Then, we have
\[
\begin{array}{rcll}
\label{h(1/2)}
h^{(2i-1)}_{k}(\frac{1}{2})&=& 0 &(i = 1, \ldots, \lfloor(k+1)/2\rfloor),\\
h^{(2i)}_{k}(\frac{1}{2})& =& \big(\!-\!\frac{1}{4}\big)^{k-i} \frac{k!(2i)!}{(k-i)!i!} & (i = 1, \ldots, \lfloor k/2\rfloor),
\end{array}
\]
\noindent
{\it Proof of Formula B.}  We obtain
 \[
 \begin{array}{rcl}
 h^{(2i-1)}_{k}(t)& =& \sum_{j=1}^{i}\frac{k!(2i-1)!}{(i-j)!(2j-1)!(k-i-j+1)!}(t^2-t)^{k-i-j+1}(2t-1)^{2j-1}\\
 h^{(2i)}_{k}(t)&=&\sum_{j=0}^{i}\frac{k!(2i)!}{(i-j)!(2j)!(k-i-j)!}(t^2-t)^{k-i-j}(2t-1)^{2j}.
\end{array}
\]
whose verification is done by induction on $i$ and is left to the reader.  \quad $\Box$

\medskip
\noindent
{\it Proof of Formula A.} Using Formula B. we calculate as follows.
\[
\begin{array}{rcl}
H_{4L}^{(2L)}(\frac{1}{2}) &=&(12L-4) h^{(2L)}_{4L-2}(\frac{1}{2}) + (4L-2)h^{(2L)}_{4L-3}(\frac{1}{2}) \\
& =&  \frac{(4L-2)!(2L)!}{(3L-2)!L!}\big(-\frac{1}{4}\big)^{3L-2}\big\{(12L-4)- 4(3L-2)\big\} \\
&= &(-1)^{L}2^{6-6L}\frac{(4L-2)!(2L)!}{(3L-2)!L!}. 
\end{array}
\]
\[
\begin{array}{rcl}
H_{4L+1}^{(2L)}(\frac{1}{2}) &= & (12L-1) h^{(2L)}_{4L-1}(\frac{1}{2}) + (4L-1) h^{(2L)}_{4L-2}(\frac{1}{2})\\
& =& \frac{(4L-1)!(2L)!}{(3L-1)!L!}\big(-\frac{1}{4}\big)^{3L-1}\big\{(12L-1)- 4(3L-1)\big\} \\
\vspace{0.2cm}
& =& (-1)^{L+1}2^{2-6L}3\frac{(4L-1)!(2L)!}{(3L-1)!L!}.\\
H_{4L+2}^{(2L)}(\frac{1}{2}) & =& (12L+2) h^{(2L)}_{4L}(\frac{1}{2}) + (4L) h^{(2L)}_{4L-1}(\frac{1}{2}) \\
& =& \frac{(4L)!(2L)!}{(3L)!L!}\big(-\frac{1}{4}\big)^{3L}\big\{(12L+2)- 4(3L)\big\}\\
\vspace{0.2cm}
& =& (-1)^{L}2^{1-6L}\frac{(4L)!(2L)!}{(3L)!L!}.\\
H_{4L+3}^{(2L)}(\frac{1}{2}) & =& (12L+5) h^{(2L)}_{4L+1}(\frac{1}{2}) + (4L+1) h^{(2L)}_{4L}(\frac{1}{2})\\
& =& \frac{(4L+1)!(2L)!}{(3L+1)!L!}\big(-\frac{1}{4}\big)^{3L+1}\big\{(12L+5)- 4(3L+1)\big\} \\
& =& (-1)^{L+1}2^{-2-6L}\frac{(4L+1)!(2L)!}{(3L+1)!L!}.
\end{array}
\]

This completes a proof of Formula and hence that of Lemma \ref{rootsH}. \qquad \quad { $\Box$}

\bigskip
\noindent
{\it Proof of } {\bf Theorem \ref{rootsD}}  According to Lemma \ref{rootsH}, let  $1=u_1>u_2>\cdots> u_{l-1}> u_l=0$ be all roots of the polynomial $H_l^{(l-4)}(t)=0$. Let $1>v_1>v_2>\cdots>v_{l-1}>0$ be the $l-1$ roots of $H_l^{(l-3)}(t)=0$ so that one has the inequalities:
\[
u_1>v_1>u_2>v_2>\cdots> u_{l-1}> v_{l-1}>u_l.
\]
On the interval $(u_{\nu+1},v_\nu)$ the functions $H_l^{(l-4)}$ and $H_l^{(l-3)}$ have the same sign for  $\nu=1,\cdots, l-1$. 
Applying Leibniz rule to \eqref{RodriguesD3}, we obtain
\[ 
\begin{array}{rcl}
\label{Leibniz}
N_{G^{dual +}_{D_{l}}}(t) = \frac{(-1)^{l-3}}{(l-2)!}\big((1-t)H^{(l-3)}_{l}(t)-(l-3)H^{(l-4)}_{l}(t) \big).
\end{array}
\] 
Therefore,  we have
\[\begin{array}{ll}
N_{G^{dual +}_{D_{l}}}(u_{\nu+1})N_{G^{dual +}_{D_{l}}}(v_\nu)=
-\frac{(1-{\nu+1})(l-3)}{((l-2)!)^2}H_l^{(l-3)}(u_{\nu+1})H_i^{(l-4)}(v_\nu)<0, 
\end{array}
\]
for $\nu=1, \ldots, l-1$. Thus, $N_{G^{dual +}_{D_{l}}}(t)=0$ has at least one root on each interval $(u_{\nu+1},v_\nu)$. Actually, there exists only one simple root on each interval, since $\deg(N_{G^{dual +}_{D_{l}}})=l$, and, therefore, together with the trivial root $t=1$, they should form the full set of roots of  $N_{G^{dual +}_{D_{l}}}(t)$.

This completes the proof of Theorem \ref{rootsD}.  
\end{proof}  
Applying Formula B to Rodrigues formula \eqref{RodriguesD}, we observe the following.

\medskip
{\bf Fact.} The skew-growth function $N_{G_{D_l}^{dual+}}(t)$ ($l\ge3$) is divisible by $2t-1$ if and only if $l=4$. 

\noindent

\section{Proof of Conjecture 3}

In this section, we prove the following theorem, which approve Conjecture 3.

\begin{theorem} 
\label{smallest}
For each series of types $P_{l} = A_l, B_l, D_l$, the smallest zero locus of $N_{G^{dual +}_{P_l}}(t)$ monotone decreasingly converge to $0$ 
as the rank $l$ tends to infinity.
\end{theorem} 
\begin{proof} 
Let us fix notation: for type $P_{l} = A_l, B_l, D_l$, let $t_{P_{l}, \nu}, \nu=1, 2, \ldots, l$, be the zeros of $N_{G^{dual +}_{P_{l}}}(t)$ in decreasing order (i.e. $1 = t_{P_{l}, 1} > t_{P_{l}, 2}  > \cdots > t_{P_{l}, l} > 0$).\\

\noindent
{\bf I.  Case for types $A_l$ and $B_l$.}  

In view of Jacobi polynomial expressions \eqref{JacobiA} and \eqref{JacobiB}, up to the first root $t_{P_l,1}=1$, the zeros $t_{l,\nu}$ and $t_{B_l,\nu}$ ($\nu=2,\cdots,l$) are equal to the zeros of $\tilde{P}_{l-1}^{(1,1)}$ and $\tilde{P}_{l-1}^{(1,0)}$, respectively. Then, the following fact is known (see \cite{[Sz]} Theorem 3.3.2.).

\begin{fact} 
The system $\{ t_{P_{l}, \nu} \}_{\nu=2}^{l}$ alternates with the system $\{ t_{P_{l+1}, \nu} \}_{\nu=2}^{l+1}$, that is,
\[
t_{P_{l+1}, \nu} > t_{P_{l}, \nu} > t_{P_{l+1}, \nu+1}, \,\,\,\,(\nu= 2, \ldots, l).
\]
\end{fact}
In particular, this implies that both sequences $\{t_{l,l}\}_{l=1}^\infty$ and $\{t_{B_l,l}\}_{l=2}^\infty$ are decreasing monotonously. On the other hand, recall a fact (see \cite{[Sz]} Theorem 6.1.2).
\begin{fact}
Let $I$ be sub-interval of $[0,1]$ of positive measure. Then, if $l$ is sufficiently large, there exists at least one $1<\nu\le l$ such that  $t_{P_l,\nu}\in I$.
\end{fact}
Applying this Fact to intervals $I=[0,\varepsilon]$ for small $\varepsilon>0$, we complete the proof of Theorem \ref{smallest}. for types $A_l$ and $B_l$. \quad  $\Box$

\bigskip
Let us give an alternative proof of Theorem \ref{smallest} for types $A_l$ and $B_l$ by describing a  distributions of all roots. Namely, we show that the density of roots is proportional to 
$dt/\sqrt{t(1-t)}$ (whose precise meaning is given in Fact \ref{Bruns}) by sandwiching the roots of types $A_l$ and $B_l$  by the roots of shifted Legendre polynomial $\widetilde P_l(t):=P_l^{(0,0)}(2t-1)$ (the shifting of Legendre polynomial $P_l^{(0,0)}(t)$).  

\begin{proposition}
1.  For type $A_l$, the following identity 
holds for $l \in\Z_{>0}$:
\begin{equation}
\label{A-Legendre}
(t N_{G^{dual +}_{{A_l}}}(t))' = N_{G^{dual +}_{{A_l}}}(t) + t N'_{G^{dual +}_{{A_l}}}(t) = (-1)^{l}\widetilde{P}_{l}(t).
\end{equation}
2.  For type $B_l$, the following identity 
holds for $l \in\Z_{\ge2}$:
\begin{equation}
\label{B-Legendre}
N_{G^{dual +}_{B_{l}}}(t) + (t/l)N'_{G^{dual +}_{B_{l}}}(t) = (-1)^{l}\widetilde{P}_{l}(t).
\end{equation}
\end{proposition}
\begin{proof}
Using the recurrence relation:
{\small $
(l+2)\widetilde{P}_{l+2}(t) = (2l+3)(2t-1)\widetilde{P}_{l+1}(t) - (l+1)\widetilde{P}_{l}(t)
$}
\ \ on shifted Legendre polynomial (c.f.\ \cite{[Sz]} 4.5.1, see also \S2), we obtain an explicit expression of the shifted Legendre polynomial for $l\in\Z_{\ge0}$.
 \begin{equation}
 \begin{array}{c}
\widetilde{P}_{l}(t) = (-1)^{l}\sum_{k=0}^{l} (-1)^{k}\frac{(l+k)!}{(l-k)!(k!)^2
}t^k.
\end{array}
\end{equation}
By comparing this expression with  the expressions in Table A of Appendix I, we obtain \eqref{A-Legendre} and \eqref{B-Legendre}.
\end{proof}
Recall a fact on the distribution of the zeros of $\widetilde{P}_l(t)$ (\cite{[Sz]} Theorem 6.21.3).
\begin{fact} (Bruns \cite{[Bru]})
\label{Bruns}
 Let $\tilde{x}_{\nu} = \tilde{x}_{l, \nu}, \nu=1, 2, \ldots, l$, be the zeros of $\widetilde{P}_{l}(t)$ in decreasing order.    
  Let $\theta_l= \theta_{l, \nu} \in (0, \pi), \nu=1, 2, \ldots, l$, be the real number defined by
\[
\cos \theta_{\nu}= 2 \tilde{x}_{\nu} - 1.
\]
Then, the inequalities hold as follows:\quad 
$ 
\frac{\nu - \frac{1}{2}}{l + \frac{1}{2}}\pi < \theta_{\nu} < \frac{\nu}{l + \frac{1}{2}}\pi \,\,\,\,\quad (\nu=1, 2, \ldots, l). 
$ 
\end{fact}

\medskip
Recall that $t_{P_{l}, \nu}, \nu=1, 2, \ldots, l$ are the zeros of $N_{G^{dual +}_{P_{l}}}(t)$ in decreasing order. Let $t'_{P_{l}, \nu}, \nu=1, 2, \ldots, l-1$, be the zeros of $N'_{G^{dual +}_{P_{l}}}(t)$ in decreasing order
and set $t'_{P_{l}, l} := 0$. From Theorem \ref{rootsP}, we see
\[
1 = t_{P_{l}, 1} > t'_{P_{l}, 1} > t_{P_{l}, 2}  > \cdots > t'_{P_{l}, l-1} > t_{P_{l}, l} > t'_{P_{l}, l}=0.
\]

\begin{proposition}
\label{distributionAB}
For type $P_{l} = A_l, B_l$, the inequalities hold as follows:
\[
1 = t_{P_{l}, 1} > \tilde{x}_{1} > t_{P_{l}, 2}  > \cdots > \tilde{x}_{l-1} > t_{P_{l}, l} > \tilde{x}_{l} > 0.
\]
\end{proposition}
\begin{proof}
We consider $2l-1$ open intervals $(t'_{P_{l}, l}, t_{P_{l}, l}), (t_{P_{l}, l}, t'_{P_{l}, l-1}), (t'_{P_{l}, l-1}, t_{P_{l}, l-1})$,\\ $
\ldots, (t_{P_{l}, 2}, t'_{P_{l}, 1}), (t'_{P_{l}, 1}, t_{P_{l}, 1})$. On the intervals $(t'_{P_{l}, l-\nu}, t_{P_{l}, l-\nu}), \nu=0, \ldots, l-1$, the polynomials $N_{G^{dual +}_{P_{l}}}(t)$ and $N'_{G^{dual +}_{P_{l}}}(t)$ have the opposite sign. Moreover, due to the identities \eqref{A-Legendre}  and \eqref{B-Legendre}, we can show $\widetilde{P}_{l}(t'_{P_{l}, l-\nu})\widetilde{P}_{l}(t_{P_{l}, l-\nu})<0, \nu=0, \ldots, l-1$. Thanks to intermediate value theorem, for the interval $(t'_{P_{l}, l-\nu}, t_{P_{l}, l-\nu})$ there exists a positive integer $i_{\nu}$ such that $\tilde{x}_{i_{\nu}} \in (t'_{P_{l}, l-\nu}, t_{P_{l}, l-\nu})$. Since the polynomial $\widetilde{P}_{l}(t)$ is of precise degree $l$, we conclude that $i_{\nu}= l-\nu$. 
\end{proof}
Combining Fact \ref{Bruns} with Proposition \ref{distributionAB}, we obtain a description of a distribution of roots of 
$N_{G^{dual +}_{{A_l}}}(t)$ and $N_{G^{dual +}_{B_{l}}}(t)$. This implies that the smallest root $t_{P_l,l}$ is given by $\cos^2(\theta_{P_l,l}/2)$ for 
$\frac{l-1/2}{l+1/2}\pi <\theta_{P_l,l}<\frac{l}{l+1/2}\pi$,   showing Theorem 5.1. 

\medskip  
\noindent
{\bf II.  Case for type $D_l$.}  

Let us give another expression of  $N_{G^{dual +}_{D_{l}}}(t)$ for $l\ge4$.
%
\begin{equation}
\begin{array}{c}
\label{JacobiD2}
N_{G^{dual +}_{D_{l}}}(t) = \frac{l-2}{2l-1}N_{G^{dual +}_{B_{l}}}(t) + \big(\frac{l+1}{2l-1}-t\big)N_{G^{dual +}_{B_{l-1}}}(t).
\end{array}
\end{equation}
\begin{proof}
From Table A, we compute the coefficient of $(-t)^k$ on the right hand side.
{\footnotesize\[
\frac{l-2}{2l-1}\cdot\frac{l(l+k-1)!}{(l-k)!k!k!} + \frac{l+1}{2l-1}\cdot\frac{(l-1)(l+k-2)!}{(l-1-k)!k!k!} + \frac{(l-1)(l+k-3)!}{(l-k)!(k-1)!(k-1)!}
\]}
{\footnotesize\[
= \frac{(l+k-3)!}{(l-k)!k!k!}\biggl\{  l(l-1)(l+k-2)+(l-2)k(k-1) \biggr\}.\,\,\,\,\,\,\,\,\,\,\,\,\,\,\,\,\,\,\,\,\,\,\,\,\,\,\,\,\,\,\,\,\,\,\,\,\,\,\,\,\,\,\,\,\,\,\,\,\,\,\,\,\,\,\,\,\,
\]}
This coincides with the coefficient of $(-t)^k$ on the left hand side in Table A.
\end{proof}
Remark that Fact 5.2.\ implies the inequlity $t_{B_l,l}< t_{B_{l-1},l-1}\le t_{B_2,2}=1/3$ for all $l\ge3$.
 Therefore, the second coefficient $\frac{l+1}{2l-1}-t$ of the formula \eqref{JacobiD2} takes positive values on the interval $ [0, t_{B_{l-1}, l-1}]$. Thanks to \eqref{JacobiD2}, we have that on the interval $[0, t_{B_{l}, l}]$ the value $N_{G^{dual +}_{D_{l}}}(t)$ is positive, and, in particular,  $N_{G^{dual +}_{D_{l}}}(t_{B_{l}, l})>0$. On the other hand, since $t_{B_l,l}<t_{B_{l-1},l-1}<t_{B_l,l-1}$ (recall Fact 6.2), we have $N_{G^{dual +}_{D_{l}}}(t_{B_{l-1}, l-1})=\frac{l-2}{2l-1} N_{G^{dual +}_{B_{l}}}(t_{B_{l-1}, l-1}) <0$. Then,  due to intermediate value theorem, we conclude that the smallest zero $t_{D_{l}, l}$ satisfies the following inequality
\[
 t_{B_{l}, l} <t_{D_{l}, l} < t_{B_{l-1}, l-1}.
\]
Since the sequence $\{t_{B_l,l}\}_{l=4}^{\infty}$ is decreasing monotone and converges to 0 as the rank $l$ tends to infinity, so is the sequence $\{t_{D_{l}, l}\}_{l=0}^{\infty}$.

This completes the proof of the case of the sequence $D_l$ and, hence, of Theorem \ref{smallest}.
\end{proof}


\section{Appendix I. } 

We give tables of explicit formulae of skew-growth functions of dual Artin monoids. Table A contains the types of three infinite series $A_l \ (l\ge1),\ B_l\ (l\ge2)$ and $D_l\ (l\ge4)$, and Table B contains the remaining exceptional types $E_6, E_7, E_8, F_4$ and $G_2$ and non-cristallographic types $H_3, H_4$ and $I_2(p)$. 

\bigskip
\centerline{\bf \large  Table A\quad }
\vspace{-0.2cm}
\[ 
\begin{array}{l}
N_{G^{dual +}_{A_l}}(t) \quad =\quad \sum_{k=0}^{l} (-1)^{k}\frac{1}{l}\binom {l}{k}\binom {l+k}{k+1}t^k,\,\,\,\,\,\,\,\,\,\,\,\,\,\,\,\,\,\,\,\,\,\,\,\,\,\,\,\,\,\,\,\,\,\,\,\,\,\,\,\,\,\,\,\,\,\,\,\,
\end{array}
\] 
\[ 
\begin{array}{l}
N_{G^{dual +}_{B_l}}(t) \quad = \quad \sum_{k=0}^{l} (-1)^{k}\binom {l}{k}\binom {l+k-1}{k}t^k,\,\,\,\,\,\,\,\,\,\,\,\,\,\,\,\,\,\,\,\,\,\,\,\,\,\,\,\,\,\,\,\,\,\,\,\,\,\,\,\,\,\,\,\,\,\,\,\,
\end{array}
\] 
\[ 
\begin{array}{l}
N_{G^{dual +}_{D_l}}(t)  \quad =\quad  \sum_{k=0}^{l} (-1)^{k} \big( \binom {l}{k}\binom {l+k-2}{k} + \binom {l-2}{k-2}\binom {l+k-3}{k}\big) t^k. 
\end{array}
\] 
\noindent
{\small The skew-growth functions for types $A_l$ and $B_l$ are obtained from F-triangles of Chapoton (\cite{[Ch]} (34) and (46)) by substituting the variables $(x,y)$ by $(-t,0)$).  The coefficients of the skew-growth functions for the type $D_l$ is due to \cite{[A-T]} \S6 Corollary 6.3.}

\bigskip
\centerline{\bf \large  Table B\quad }
\vspace{-0.3cm}
\[
\begin{array}{rcl}
N_{G^{dual +}_{E_6}}(t)& = & 1 - 36 t + 300 t^2 - 1035 t^3 + 1720 t^4 - 1368 t^5 + 418 t^6 \\
N_{G^{dual +}_{E_7}} (t)& = &1\! -\! 63 t\! +\! 777 t^2 \! -\! 3927 t^3 \!+\! 9933 t^4 \!-\! 13299 t^5 \!+\! 9009 t^6 \!-\! 2431t^7 \\
N_{G^{dual +}_{E_8}} (t)& = &\!\!\!\!\!\!{\scriptsize1\!\! -\!\! 120 t \!\! +\!\! 2135 t^2 \!\!\! -\!\! 15120 t^3 \!\!\! +\!\! 54327 t^4 \!\!\!- \!\! 108360 t^5 \!\!\! + \!\!121555 t^6 \!\!\! - \!\! 71760 t^7 \!\!\! + \!\! 17342 t^8}\\
N_{G^{dual +}_{F_4}} (t)& = & 1 - 24 t + 101 t^2 - 144 t^3 + 66 t^4 \\
N_{G^{dual +}_{G_2}} (t)& = & 1 - 6 t + 5 t^2 \\ 
N_{G^{dual +}_{H_3}} (t)& = & 1 - 15 t + 35 t^2 - 21 t^3 \\
N_{G^{dual +}_{H_4}} (t)& = & 1 - 60 t + 307 t^2 - 480 t^3 + 232 t^4 \\
N_{G^{dual +}_{I_2(p)}} (t)& = & 1 - p t + (p-1) t^2 
\end{array}
\]
\noindent
{\small 
The skew-growth functions for exceptional types $E_6, E_7, E_8, F_4$ and $G_2$ are obtained from the formula §5 (9) in \cite{[A-T]}, to which one applies the data in §6 Table 1. 
The skew-growth functions for non-crystallographic types $H_3, H_4$ and $I_2(p)$ are obtained as special cases of the data of \cite{[Ar]} \S5.3. Figure 5.14, provided by 
F.\ Chapoton.
}

\section{Appendix II}
 \medskip
 \noindent
 {\bf Zero loci of $N_{G^{dual +}_{P}}(t)=0$ in the complex plane for type $P = A_{20}, B_{20}, D_{20}$ and $E_8$. The zeroes are indicated by $+$.} 

\bigskip
\noindent
\quad\  \includegraphics{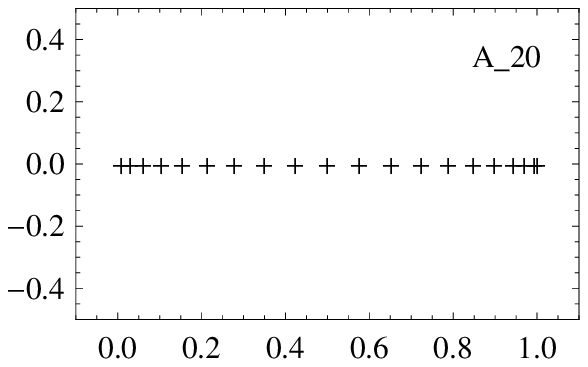}
\includegraphics{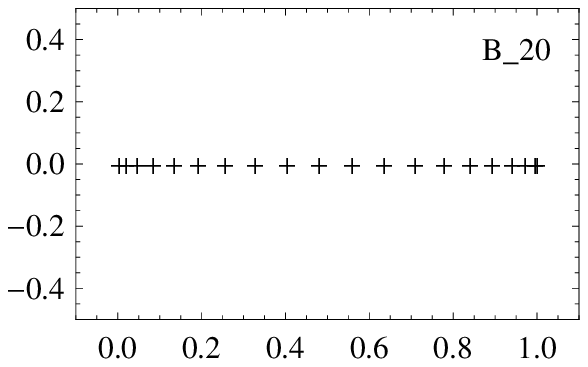}\\

\smallskip
\!\!\! \includegraphics{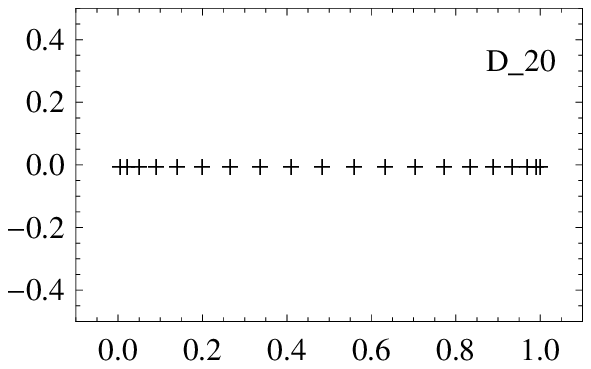}
\includegraphics{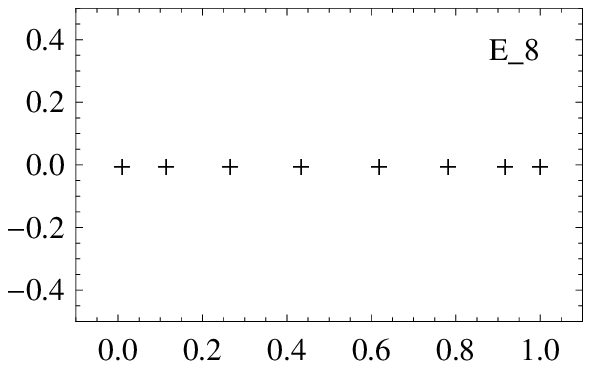}

\section {Appendix III}
In the present Appendix, we first recall David Bessis's description \cite{[Be]} of a dual Artin monoid $G^{dual +}_{P}$ and its relationship with the non-crossing partition. 
Then, we identify the skew-growth function of the dual Artin monoid with the characteristic polynomial of the lattice of non-crossing partitions of type $P$. We use the notation of \S1 freely.

   Let $(\overline{G}_P,S)$ be a Coxeter system of finite type $P$ and let $T\subset \overline{G}_P$ the set of all reflections as in \S1. 
   For an element $g\in \overline{G}_P$, the length $l_T(g)\in \Z_{\ge0}$ is defined as the minimal length of words expressing $g$ by $T$. Clearly the length function $l_T$ is subadditive. For $g,h\in \overline{G}_P$, we say $g$ divides $h$ from the left (with respect to $T$), and denote $g|^l_Th$ or $g\le_T h$, if $l_T(g)+l_T(g^{-1}h)=l_T(h)$. Similarly, we say $g$ divides $h$ from the right and denote $g|^r_Th$ or $h\ge_T g$, if $l_T(g)+l_T(hg^{-1})=l_T(h)$. Since $T$, and hence $l_T$, is invariant by conjugation, one has $g \le_T g' \Leftrightarrow g'\ge_T g$ for any $g,g'\in \overline G_P$. 
So, we define for any $h\in \overline{G}_P$ the interval 
   \begin{equation}
   [1,h]_T:=\{g\in \overline{G}_P \ |\  g|^l_Th\} =\{g\in \overline{G}_P \ |\ g|^r_Th\} .
   \end{equation}
On the interval, a pre-monoid structure\footnote
{A set $A$ will be called a pre-monoid if the product is defined from a subset of $A\times A$ to $A$ which satisfies: i) unitarity: $\exists 1\in A$ such that for any $g\in A$, $1g$ and $g1$ are defined as $g$, and ii) associativity: products $gg'$ and $(gg')g"$ are defined if and only if products $g'g"$ and $g(g'g")$ are defined and $(gg')g"=g(g'g")$ for any $g,g',g"\in A$
(see \cite{[Be]}).}
is defined: namely, for a pair $(g,g')\in [1, h]_T\times[1, h]_T$ a product $g\cdot g':=gg'$ is defined in the pre-monoid only when $gg'\in [1, h]_T$ and $l_T(gg')=l_T(g)+l_T(g')$. 
Let $M: \mathcal{C}_{preMonid} \to \mathcal{C}_{Monid}$ be the left adjoint functor of the natural embedding $\mathcal{C}_{Monoid} \to \mathcal{C}_{preMonoid}$ of the category of monoids to the category of pre-monoids.\footnote
{Let $A$ be a pre-monoid. Then, $M(A)$ is explicitly given by the monoid generated by elements of $A$ and defined by the relations $g\cdot g' =gg'$ for all $g,g'\in A$ whenever their product is defined in the pre-monoid $A$ (here we denote by $g\cdot g'$ also the product of $g,g'$ in the monoid $M(A)$). Similarly, we denote by $G:  \mathcal{C}_{Monoid} \to \mathcal{C}_{Group}$ the left adjoint functor of the natural embedding $\mathcal{C}_{Group} \to \mathcal{C}_{Monoid}$ of the category of groups to the category of monoids.
}
%
The monoid $M([1, h]_T)$ is cancellative in the sense that for $m\in M([1, h]_T)$ and $g,g'\in [1, h]_T$ if $g\cdot m=g'\cdot m$ or $m\cdot g=m\cdot g'$ holds in $M([1, h]_T)$, then $g=g'$ holds in $[1,h]_T$ ({\it Proof.} The assumption implies
$\bar{g}\bar{m}=\bar{g}'\bar{m}$ or $\bar{m}\bar{g}=\bar{m}\bar{g}'$  holds in $\overline{G}_P$ (here, we denote by $\bar{m}$, $\bar{g}$ and $\bar{g}'$ the images in $\overline{G}_P$ of  $m$, $g$ and $g'$, respectively), and, therefore $g=g'$ holds in 
$\overline{G}_P$ and hence in $[1,h]_T$. Here, the natural composition: $[1, h]_T\! \to\! M([1, h]_T) \to \overline{G}_P$ is injective).

Since the Coxeter diagram associated with the simple generator system $S$ is a tree, we have a bipartite decomposition $S=L\coprod R$ (unique up to a transposition of $L$ and $R$) such that elements in $L$ (resp. $R$) are mutually commutative.  Then a {\it Coxeter element} $c$ in the Artin monoid $G^+_P$ is defined to be the product of the form $c=c_Lc_R$ where $c_L:=\prod_{a\in L}a$ and $c_R:=\prod_{a\in R}a$. The projection $\bar c$ of $c$ in the Coxeter group $\overline{G}_P$ is the Coxeter element in the classical sense. Clearly $\bar c$ is divisible (in the sense of $|_T$) by all elements of $T$ so that $T\subset [1,\bar{c}]_T$. We call two reflections $r,s\in T$ a {\it non-crossing pair} if $\bar r \bar s\le_T \bar{c}$ holds. Then we set $s^*:=rsr^{-1}\in T$. 

\begin{theorem} (Bessis \cite{[Be]}).  
\noindent
{\bf 1.} The monoid $M([1,\bar c]_T)$ is generated by the set $T$ and defined by the quadratic relations \eqref{dual relation}
for all non crossing pairs $r,s\in T$. 

\medskip
\noindent
{\bf 2.} Set ${\bf T}:=\cup_{k\in \Z} (c^k {\bf S} c^{-k}) \subset G_P$. Then the natural projection $\pi: G_P\to \overline{G}_P$ induces a bijection ${\bf T} \simeq T$ and, further, an isomorphism from the Artin group $G_P$ to the group 
$G(M([1,\bar c]_T))$ 
$\simeq \big<  T \mid \text{relations given in {\bf 2.}}\big>_{gr}$ (see Footnote 4 for $G(\cdot)$). 

\end{theorem}
\begin{proof}   See [Be. 2.1.4, 2.2.2, 2.2.5, 2.3.3].
\end{proof} 
The relations \eqref{dual relation} are called the {\it dual Artin braid relations}. We remark that 
if a pair $r$ and $s$ is {\it non-crossing}, then the pair $s$ and  $^*\!r\!\!:=s^{-1}rs$ is also non-crossing (since $\bar s\ ^*\!\bar r =\bar r\bar s\le_T \bar c$). Therefore, $s\cdot ^*\!r=r\cdot s$ is also a dual relation.
Let us  show that the monoid $M([1,c]_T)$, which we shall call the {\it dual Artin monoid}, is a submonoid of the Artin group. To see that, we use a distinguished property of the monoid.

\begin{theorem} (Bessis \cite{[Be]}).  
{\bf 1.}  The pre-monoid $[1,\bar c]_T$ is a lattice.

\medskip
\noindent
{\bf 2.}  $M([1,\bar c]_T)$ is a lattice, and hence it is a Garside monoid.

\medskip
\noindent
{\bf 3.}   The dual Artin monoid $G^{dual +}_{P}:=M([1,\bar c]_T)$ is naturally isomorphic to the submonoid of the Artin group $G_P$ generated by ${\bf T}$.
\end{theorem}
\begin{proof}
{\bf 1.} The original proof by Bessis [Be, Fact 2.3.1] used the classification of finite reflection groups and was shown case by case. A case free proof is given by [B-W].

{\bf 2.}  To show the lattice property of $M([1,\bar c]_T)$ from that of $[1,\bar c]_T$ is formal (see 
[B-S,D, Be, Theorem 2.3.2]).

{\bf 3.}  It is a general property that a Garside monoid $M$ is embedded into $G(M)$  (see [B-S, Be, Theorem 2.3.2]).
\end{proof}

Since the lattice $[1,\bar c]_T$, for the case of type $A_l$, coincides with  the lattice of {\it non-crossing partitions} (Kreweras \cite{[K]},  Birman-Ko-Lee \cite{[B-K-L]}), the lattice $[1,\bar c]_T$ for the case of type $P$ is called a {\it non-crossing partition lattice of  types} $P$.

\medskip
We are interested in studying the skew-growth function $N_{G^{dual +}_{P},\deg}(t)$ (\ref{skew-growth})  of the dual Artin monoid with respect to the degree map \eqref{degree map}. 
Let us show 

\begin{fact} 
\label{skew=character}
The skew-growth function $N_{G^{dual +}_{P},\deg}(t)$ coincides with the characteristic polynomial $\chi_{[1,\bar c]_T}(t)$ of the non-crossing partition lattice $[1,\bar c]_T$ of type $P$. 
\end{fact}
\begin{proof} Recall that the characteristic polynomial of a finite graded poset $L$ with the minimal element, denoted by $\hat 0$, is the generating function of the M\"obius function:
\[
\begin{array}{c}
\chi_L(t)=\sum_{x\in L}\mu(\hat 0,x)t^{\deg(x)}
\end{array}
\]
(\cite{[St]} (3.39)).  Suppose that $L$ carries a pre-monoid structure and that the poset structure on $L$ is induced from the (left-)divisibility relation of the pre-monoid structure (eg.\ $[1,\bar c]_T$ and $[1,\Delta]_S$, where 1 plays the role of $\hat 0$) and let $I\subset L$ be the minimal generator system of the pre-monoid. We consider the graded ring $\Z[[L]]:=(\Z\cdot M(L))\hat{}$ of the monoid $M(L)$. We note also that $1\in M(L)$  (corresponding to the minimal element of $L$) is the only element whose degree is equal to 0.


 Inside the algebra $\Z[[L]]$, we obtain the monoid theoretic inversion formula :
\[
\begin{array}{c}
     \big(\sum_{J\subset I} lcm(J)\big) \cdot \big(\sum_{x\in M(L)} x\big) =1.
\end{array}
\]
(\cite{[Sa2]} Cor. 5.4). Applying to this the operator $t^{\deg}$, and by setting  $P_{M(L),\deg}(t):=\sum_{x\in M(L)}t^{\deg(x)}$, we obtain the inversion formula $P_{M(L),\deg}(t)\cdot N_{M(L),\deg}(t)=1$.
On the other hand, Cartier-Foata theory \cite{[C-F]} gives the (combinatorial theoretic) inversion formula $P_{M(L),\deg}(t)\cdot \sum_{x\in L}\mu(\hat 0,x)t^{\deg(x)}=1$. Thus, combining the both inversion formulae, we conclude the equality $\chi_{[1,\bar c]_T}(t)=N_{G_P^{dual +}}(t)$.
\end{proof}
\begin{remark}
The equality $\chi_{[1,\bar c]_T}(t)=N_{G_P^{dual +}}(t)$ implies the explicit formula:
\[
\begin{array}{c}
\sum_{J\subset I, lcm_l(J) = x}(-1)^{\# J}=\mu(\hat 0, x).
\end{array}
\]
for all $x \in L$. 
This can be shown directly by induction on $\deg(x)$. 
\end{remark}
\noindent
({\it Proof.} First, for $\deg(x)=0$ the statement is true. Next, by induction hypothesis, we have
$
\sum_{J\subset I, lcm_l(J) \lneqq x}(-1)^{\# J}=\sum_{z \lneqq x}\mu(\hat 0, z).
$
\quad  On the other hand, we have
$
\sum_{J\subset I, lcm_l(J) \leqq x}(-1)^{\# J}=0.
$ 
This implies the explicit formula.)

\bigskip
\emph{Acknowledgement.}~
The authors express their gratitudes to James Wallbridge and Mikhail Kapranov for their careful reading of the  manuscript and for their suggestions of improvements.

This researsh was supported by World Premier International Research Center Initiative (WPI Initiative), MEXT, Japan, and by JSPS KAKENHI Grant Number 25247004.

\medskip
\noindent
{\footnotesize
Kavli IPMU (WPI), UTIAS, the University of Tokyo,  Kashiwa, Chiba 277-8583, Japan


E-mail address: chamarims@yahoo.co.jp

E-mail address: kyoji.saito@ipmu.jp
 }

\end{document}